\documentclass[a4paper,10pt]{amsart}

\newtheorem{theo}{Theorem}[section]

\newtheorem{prop}[theo]{Proposition}
\newtheorem{lem}[theo]{Lemma}
\newtheorem{cor}[theo]{Corollary}
\newtheorem{conj}[theo]{Conjecture}

\newtheorem{rema}[theo]{Remark}

\def \kbar {{\overline k}}

\def \Romannumeral #1 {\expandafter\uppercase\expandafter {\romannumeral #1} }

\def \pic {{\rm {Pic}}}

\def \calf {{\mathcal F}}
\def \calg {{\mathcal G}}
\def \spec {{\rm{Spec\,}}}

\def \dim {{\rm{dim\,}}}

\def\ra{\rightarrow}

\def \Z {{\bf Z}}
\def \Q {{\bf Q}}
\def \F {{\bf F}}

\def \RR {{\bf R}}
\def \C {{\bf C}}

\def \X {{\mathcal X}}
\def \Y {{\mathcal Y}}
\def \V {{\mathcal V}}
\def \W {{\mathcal W}}
\def \T {{\mathcal T}}
\def \L {{\mathcal L}}
\def \Het {H_{\mbox{\rm\scriptsize\'et}}}

\def\smallsquare{\vbox{\hrule\hbox{\vrule height 1 ex\kern 1 ex\vrule}\hrule}}

\def\ICY{{\rm IC}_{\mathcal Y}}
\def\ICXB{{\rm IC}_{{\mathcal X}^B}}

\usepackage{amscd}
\usepackage{amssymb}
\usepackage{amsmath}
\usepackage{pb-diagram}

\everymath{\displaystyle}

\DeclareFontFamily{U}{wncy}{}
\DeclareFontShape{U}{wncy}{m}{n}{%
   <5>wncyr5%
   <6>wncyr6%
   <7>wncyr7%
   <8>wncyr8%
   <9>wncyr9%
   <10>wncyr10%
   <11>wncyr10%
   <12>wncyr6%
   <14>wncyr7%
   <17>wncyr8%
   <20>wncyr10%
   <25>wncyr10}{}
\DeclareMathAlphabet{\cyrille}{U}{wncy}{m}{n}

\def \R{{\bf R}}

\def \calf{{\mathcal F}}

\def \calg{{\mathcal G}}

\usepackage{palatino}

\title{A generalization of Beilinson's geometric height pairing}
\author{Damian R\"ossler and Tam\'as Szamuely}
\address{Mathematical Institute, University of Oxford, Andrew Wiles Building, Radcliffe Observatory Quarter, Woodstock Road, Oxford OX2 6GG, United Kingdom}
\email{damian.rossler@maths.ox.ac.uk}
\address{Dipartimento di Matematica, Universit\`a di Pisa, Largo Bruno Pontecorvo 5, 56127 Pisa, Italy and Alfr\'ed R\'enyi Institute of Mathematics, Hungarian Academy of Sciences, Re\'altanoda utca 13--15, H-1053 Budapest, Hungary}
\email{tamas.szamuely@unipi.it}
\date{\today}
\begin{document}
\maketitle

\begin{abstract} In the first section of his seminal paper on height pairings, Beilinson constructed an $\ell$-adic height
pairing for rational Chow groups of homologically trivial cycles of complementary codimension on
smooth projective varieties over the function field of a curve over an algebraically closed field, and asked about an generalization to higher dimensional bases. In this paper we answer Beilinson's question by constructing a pairing for varieties defined over the function field of a smooth variety $B$ over an algebraically closed field, with values in the second $\ell$-adic cohomology group of $B$. Over $\C$ our pairing is in fact $\Q$-valued, and in general we speculate about its geometric origin.\end{abstract}

\section{Introduction}

Let $k$ be an algebraically closed field, $B$ a smooth integral $k$-scheme of finite type with function field $K=k(B)$, and $X$ a smooth projective integral $K$-variety of dimension $d$. Fix a prime $\ell$ invertible in $k$. For an integer $i\geq 0$ we denote by $CH^i(X)_\Q$ the Chow group of codimension $i$ cycles on $X$ tensored by $\Q$ and by $CH^i_{\rm hom}(X)_\Q$ the subspace of homologically trivial cycles, i.e. the kernel of the $\ell$-adic \'etale cycle map
$$
CH^i(X)_\Q\to \Het^{2i}(X_{\overline K}, \Q_\ell(i))
$$
where $\overline K$ stands for an algebraic closure of $K$.

In the first section of his seminal paper \cite{beilinsonhp}, Beilinson worked in the case where $B$ is a smooth proper curve and constructed (unconditionally) an $\ell$-adic height pairing
\begin{equation}\label{dim1pairing}
CH^p_{\rm hom}(X)_\Q\otimes CH^q_{\rm hom}(X)_\Q\to \Q_\ell
\end{equation}
for  $p+q=d+1$. This served as a motivation for his (conditional) construction of the height pairing in the number field case.

In this article we extend Beilinson's work to a not necessarily proper base $B$ of arbitrary dimension, solving the problem stated at the end of subsection (1.1) of \cite{beilinsonhp}.

\begin{theo}\label{mainthm} In the situation above and for $p+q=d+1$ there exists a pairing
\begin{equation}\label{pairing}
CH^p_{\rm hom}(X)_\Q\otimes CH^q_{\rm hom}(X)_\Q\to \Het^{2}(B,\Q_\ell(1))
\end{equation}
which for a smooth proper $B$ of dimension 1 coincides with Beilinson's pairing (\ref{dim1pairing}) after composing with the trace isomorphism $ \Het^{2}(B,\Q_\ell(1))\stackrel\sim\to\Q_\ell$ of Poincar\'e duality.

There is also a similar pairing in the case where $k$ is algebraic over a finite field instead of being algebrically closed.
\end{theo}

The proof of the theorem is based on Beilinson's original philosophy. In Section 2 we construct a cohomological pairing on the `perverse parts' of those cohomology groups that contain classes of the relevant homologically trivial cycles. Then we prove that these cycle classes actually lie in the perverse part. We offer two methods for this, in Sections \ref{secfinite} to \ref{secdec}. The first one starts with the finite field case where weight arguments can be invoked, and then applies a specialization argument. The second one relies on the decomposition theorem for perverse sheaves of geometric origin but not on weight arguments. \smallskip

Our proof based on the decomposition theorem works equally well in the topological setting when $k=\C$. In this case we obtain a result with $\Q$-coefficients whose proof will be sketched at the end of Section \ref{secdec}:

\begin{theo}\label{complextheo} Assume $k=\C$. For $p+q=d+1$ there exists a pairing with values in Betti cohomology
\begin{equation}\label{pairing}
CH^p_{\rm hom}(X)_\Q\otimes CH^q_{\rm hom}(X)_\Q\to H^{2}(B(\C),\Q(1))
\end{equation}
which composed with the natural map $H^{2}(B(\C),\Q(1))\to \Het^{2}(B,\Q_\ell(1))$ coincides with the pairing of Theorem \ref{mainthm}.
\end{theo}

We conjecture that our pairing is of motivic origin:

\begin{conj}\label{mainconj} The pairing of Theorem \ref{mainthm} comes from a pairing
$$
CH^p_{\rm hom}(X)_\Q\otimes CH^q_{\rm hom}(X)_\Q\to {\rm Pic}(B)_\Q
$$
followed by the cycle map.
\end{conj}

In Section \ref{secgoodred} we prove the conjecture under the strong assumption that $X$ extends to a smooth proper $B$-scheme. This is done using a simple construction involving the intersection product.
In the bad reduction case there is some evidence for the conjecture when $X$ is an abelian variety: the construction of Moret-Bailly (\cite{moret}, \S III.3) gives the required pairing in the case $p=1$, in fact over an arbitrary field $k$ and only assuming $B$ to be normal.

In the case of a proper $B$ of dimension $b$ we expect that the pairings of Conjecture \ref{mainconj} yield non-degenerate pairings after composing with the degree map ${\rm Pic}(B)_\Q\to\Q$ associated with a fixed ample line bundle $\mathcal L$ on $B$ (i.e. the map obtained by intersecting a class in ${\rm Pic}(B)_\Q\cong CH_{b-1}(B)_\Q$ with the $(b-1)$-st power of the class of $\mathcal L$ in $CH_{b-1}(B)_\Q$ and then taking the degree of the resulting zero-cycle class).

Furthermore, we expect pairings as in Conjecture \ref{mainconj} to exist over an arbitrary perfect field but at the moment we are not able to construct their cohomological realization in this more general setting. However, we offer an Arakelovian counterpart of the conjecture in Section \ref{secarak} below.

Inspired by a lecture on a preliminary version of this note, Bruno Kahn \cite{kahn} has defined a certain subgroup $CH^p(X)^{(0)}\subset CH^p(X)$ and constructed a pairing
$$
CH^p(X)^{(0)}_\Q\otimes CH^q_{\rm hom}(X)^{(0)}_\Q\to {\rm Pic}(B)_\Q
$$
purely by cycle-theoretic manipulations, over an arbitrary perfect field. He conjectures that his subgroup $CH^p(X)^{(0)}$ is in fact the subgroup of numerically trivial cycles and therefore contains $CH^p_{\rm hom}(X)$. This is strong evidence in support of our conjecture.

We thank Bruno Kahn and Klaus K\"unnemann for very helpful discussions. A preliminary version of Kahn's paper \cite{kahn} was also useful when writing up the details of the proof of Proposition \ref{homeq} below.\smallskip

\noindent{\em A word on conventions:} Except for the last section, by `variety' we shall always mean an integral separated scheme of finite type over a field. When working with perverse sheaves in the $\ell$-adic setting we shall use $\Q_\ell$-coefficients instead of ${\overline \Q}_\ell$-coefficients in order to remain in line with Beilinson's original setup. This is justified by (\cite{bbd}, Remark 5.3.10 and the remarks in 4.0).

\section{The cohomological pairing}\label{seccons}

We keep the notation from the introduction, except that the base field $k$ is allowed to be an arbitrary perfect field. We spread out the morphism $X\to\spec K$ to a \hbox{smooth} projective morphism $\pi:{\mathcal X}\to U$ for a suitable affine open subscheme $U\subset B$, and denote by $j$ the inclusion map $U\to B$. This section is devoted to the construction of a pairing
\begin{equation}\label{cohpairing}
\Het^{1-b}(B,j_{!*}\R^{2p-1}\pi_*\Q_\ell(p) [b])\times \Het^{1-b}(B, j_{!*}\R^{2q-1}\pi_*\Q_\ell(q)[b])\to \Het^{2}(B, \Q_\ell(1))
\end{equation}
where $b:=\dim B=\dim U$ and $j_{!*}$ is, as usual, the intermediate extension functor familiar from the theory of perverse sheaves (notice that since $\R^{2p-1}\pi_*\Q_\ell(p)$ is  locally free on $U$, its shift by $b$ is a perverse sheaf on $U$).

We proceed in two steps.\medskip

\noindent {\em Step 1.} Assume given a lisse sheaf $\mathcal F$ on an open subscheme $j:\, U\hookrightarrow B$. Consider the duality pairing
$$
\calf [b]\otimes^{\bf L} \R Hom(\calf [b], \Q_\ell(b)[2b])\to \Q_\ell(b)[2b].
$$
Such a pairing corresponds to a morphism $$\calf[b]\to D(\R Hom(\calf [b], \Q_\ell(b)[2b]))$$ in $D^b(U,\Q_\ell)$, where $D$ is the dualizing functor corresponding to the dualizing complex $\Q_\ell(b)[2b]$ on $U$. Since $j_{!*}\Q_\ell(b)[2b]=\Q_\ell(b)[2b]$ by smoothness of $B$ and the intermediate extension functor $j_{!*}$ is interchangeable with $D$ (see \cite{kiwe}, Corollary III.5.3), applying $j_{!*}$ gives a map
$$j_{!*}\calf[b]\to D(j_{!*}\R Hom(\calf [b], \Q_\ell(b)[2b])),$$
whence finally a duality pairing of perverse sheaves
$$
j_{!*}\calf [b]\otimes^{\bf L} j_{!*}\R Hom(\calf [b], \Q_\ell(b)[2b])\to \Q_\ell(b)[2b].
$$
\medskip

\noindent {\em Step 2.} We apply the above to $\calf:=\R^{2p-1}\pi_*\Q_\ell$. In this case fibrewise Poincar\'e duality gives $\R Hom(\R^{2p-1}\pi_*\Q_\ell , \Q_\ell)\cong \R^{2q-1}\pi_*\Q_\ell(d)$ with our convention $p+q=d+1$, so that
$$
\R Hom(\R^{2p-1}\pi_*\Q_\ell [b], \Q_\ell(b)[2b])\cong \R^{2q-1}\pi_*\Q_\ell(b+d)[b].
$$
Plugging this into the pairing of Step 1 we obtain a pairing
$$
j_{!*}\R^{2p-1}\pi_*\Q_\ell [b]\otimes^{\bf L} j_{!*}\R^{2q-1}\pi_*\Q_\ell(b+d)[b]\to \Q_\ell(b)[2b]
$$
and, after twisting by $\Q_\ell(p+q-d-b)=\Q_\ell(1-b)$, a pairing
$$
j_{!*}\R^{2p-1}\pi_*\Q_\ell(p) [b]\otimes^{\bf L} j_{!*}\R^{2q-1}\pi_*\Q_\ell(q)[b]\to \Q_\ell(1)[2b]
$$
in the bounded derived category of $\Q_\ell$-sheaves on $B$. Passing to cohomology, we finally obtain the announced pairing (\ref{cohpairing}).
\medskip

In dimension 1, the groups in the arguments of the pairing (\ref{cohpairing}) are those considered in Lemma 1.1.1 of \cite{beilinsonhp}. Note that there are
morphisms
\begin{equation}\label{pervpart1} \Het^{1-b}(B,j_{!*}\R^{2p-1}\pi_*\Q_\ell(p) [b])\to \Het^{1}(U,\R^{2p-1}\pi_*\Q_\ell(p))\end{equation}
and
\begin{equation}\label{pervpart2}\Het^{1-b}(B, j_{!*}\R^{2q-1}\pi_*\Q_\ell(q)[b])\to \Het^{1}(U,\R^{2q-1}\pi_*\Q_\ell(q))\end{equation}
induced by the natural map from $j_{!*}$ to $\R j_*$. We now prove that, again in accordance with the dimension 1 case, these maps are in fact injective. This will be a special case of the following more general statement.

\begin{prop}\label{keylem}
For a perverse sheaf $\calf$ on $U$ the natural map
\begin{equation}\label{morkeylem}
\Het^{1-b}({ B}, { j}_{!*}\calf)\to \Het^{1-b}({ B}, \RR{ j}_{*}\calf)
\end{equation}
is injective.
\end{prop}

The proof of the proposition is based on a vanishing lemma.  If $i$ is the inclusion map of the closed complement $Z$ of $U$ in $X$, set
\begin{equation}\label{defcalg}
\calg:={ i}_*{}^p\tau_{\geq 0}{ i}^*\RR{ j}_*\calf
\end{equation}
where ${}^p\tau_{\geq 0}$ denotes, as usual, the perverse truncation functor. Recall (e.g. from \cite{kiwe}, Lemma III.5.1) that $\calg$ sits in a distinguished triangle
\begin{equation}\label{calgg}
{ j}_{!*}\calf\to  \RR{ j}_{*}\calf\to \calg\to { j}_{!*}\calf[1].
\end{equation}
Since ${ j}_{!*}\calf$ and   $\RR{ j}_{*}\calf$ are perverse sheaves (the first by definition, the second by \cite{kiwe}, Corollary III.6.2), so is $\calg$.

\begin{lem}\label{lemkeylem}
With notation as above, we have
$$
\Het^{c}({ B}, \calg)=0
$$
for all $c<-b+1$, where $b=\dim B$ as before.
\end{lem}

\begin{proof} The complex ${}^p\tau_{\geq 0}{ i}^*\RR{ j}_*\calf$ lives  on the closed subscheme $Z$ which is of dimension at most $b-1$. Thus by (\cite{kiwe}, Lemma III.5.13) its cohomology sheaves are zero in degrees $<-b+1$. Since $i_*$ is an exact functor, the cohomology sheaves ${\mathcal H}^t(\calg)$  also vanish for $t<-b+1$. Thus in the hypercohomology spectral sequence
$$
E_2^{st}=\Het^s({ B}, {\mathcal H}^t(\calg))\Rightarrow \Het^{s+t}({ B}, \calg)
$$
the terms $E_2^{st}$ are 0 for $s<0$ or $t<-b+1$, hence for $s+t<-b+1$.
\end{proof}

\begin{proof}[Proof of Proposition \ref{keylem}] The distinguished triangle (\ref{calgg}) induces an exact sequence of cohomology groups
$$
\Het^{-b}({ B},\calg)\to\Het^{1-b}({ B}, { j}_{!*}\calf)\to \Het^{1-b}({ B}, \RR{ j}_{*}\calf)
$$
where the first term vanishes  by the lemma.
\end{proof}

The strategy of the proof of Theorem \ref{mainthm} is now as follows. Consider a homologically trivial cycle class $\alpha^p\in CH^p_{\rm hom}(X)_\Q$. Extending a representative $z$ of $\alpha^p$ to a cycle $z_U$ on a model $\pi:\,\X\to U$  of $X$ over $U$ sufficiently small, we obtain a cohomology class ${\rm cl}(z_U)\in \Het^{2p}(\X, \Q_\ell(p))$. Since $z$ is homologically trivial, the proper smooth base change theorem implies that the image of ${\rm cl}(z_U)$ in $\Het^{2p}(\X_{\overline x}, \Q_\ell(p))$ is trivial for every geometric point $\overline x$ of $U$. Thus it is in the kernel of the natural map
\begin{equation}\label{leraymap}
\Het^{2p}(\X, \Q_\ell(p))\to \Het^{0}(U, \R^{2p}\pi_*\Q_\ell(p)),
\end{equation}
whence a class $\alpha^p_U\in\Het^{1}(U,\R^{2p-1}\pi_*\Q_\ell(p))$.
Similarly, a class $\alpha^q$ in $CH^q_{\rm hom}(X)_\Q$ has a representative
$\alpha^q_U\in\Het^{1}(U,\R^{2q-1}\pi_*\Q_\ell(q))$ for $U\subset B$ sufficiently small.

For later use we spell this out in more detail. Noting that
$$
\Het^{2p}(\X, \Q_\ell(p))\cong \Het^{2p}(U,\R\pi_*\Q_\ell(p))\cong \Het^{2p}(U, \tau_{\leq 2p}\R\pi_*\Q_\ell(p))
$$
where $\tau_\leq$ denotes the sophisticated truncation functor on $D^b(U, \Q_\ell)$, the distinguished triangle
$$
\tau_{\leq 2p-1}\R\pi_*\Q_\ell(p)\to \tau_{\leq 2p}\R\pi_*\Q_\ell(p)\to \R^{2p}\pi_*\Q_\ell(p)[-2p]\to \tau_{\leq 2p-1}\R\pi_*\Q_\ell(p)[1]
$$
identifies the kernel of (\ref{leraymap}) with  $\Het^{2p}(U, \tau_{\leq 2p-1}\R\pi_*\Q_\ell(p))$,
which gives rise to the class $\alpha^p_U$ by applying the natural map
$
\tau_{\leq 2p-1}\R\pi_*\Q_\ell(p)\to \R^{2p-1}\pi_*\Q_\ell(p)[-2p+1].
$

In view of the injectivity of the maps (\ref{pervpart1}) and (\ref{pervpart2}) proven in Proposition \ref{keylem} the following statement makes sense.

\begin{prop}\label{mainprop} Assume $k$ is algebraically closed or algebraic over a finite field.

The classes $\alpha^p_U$ and $\alpha^q_U$ lie in the subgroups $$\Het^{1-b}(B,j_{!*}\R^{2p-1}\pi_*\Q_\ell(p) [b])\subset \Het^{1}(U,\R^{2p-1}\pi_*\Q_\ell(p))$$ and $$\Het^{1-b}(B,j_{!*}\R^{2q-1}\pi_*\Q_\ell(q) [b])\subset \Het^{1}(U,\R^{2q-1}\pi_*\Q_\ell(q)),$$ respectively.
\end{prop}

Theorem \ref{mainthm} immediately follows from this proposition by applying the pairing (\ref{cohpairing}) to the classes $\alpha^p_U$ and $\alpha^q_U$; by a standard argument the resulting class in $\Het^{2}(B, \Q_\ell(1))$ only depends on the classes $\alpha^p$  and $\alpha^q$.\medskip

In the case of a finite base field we'll see in the next section using a weight argument that the maps (\ref{pervpart1}) and (\ref{pervpart2}) are in fact isomorphisms, so Proposition \ref{mainprop} will be obvious for $k$ algebraic over a finite field.
We shall offer two proofs of Proposition \ref{mainprop} over algebraically closed fields. The first one, given in the next two sections, will proceed by reduction to the finite field case using a specialization argument. The second one will use the decomposition theorem for perverse sheaves. Note that the proof of this theorem in \cite{bbd} also proceeds by a specialization argument and weight arguments over a finite base field, so one may argue that the two arguments are not very different. However, over $\C$ there are other proofs of the decomposition theorem that avoid weight arguments (see \cite{decmi}, \cite{saito}) and therefore, combined with a base change argument as in Lemma \ref{bcdiag} below, we do obtain a different proof in characteristic 0.

\section{Proof of \ref{mainprop} over finite fields}\label{secfinite}

In this section we first consider the case of a finite base field $k$. We begin with
a statement on weights. Needless to say, its proof is based on Deligne's fundamental results \cite{weil2} on the Weil conjecture which we cite through (\cite{bbd}, (5.1.14)).

\begin{lem}\label{keylem2}
Assume that $k$ is a finite field, and  $\calf$ is a pure perverse sheaf on $U$ of weight $w$. Consider the base change $\overline B$ of $B$ to the algebraic closure $\overline k$, and the pullback $\overline\calg$ of the sheaf $\calg$ defined in (\ref{defcalg}) to $\overline B$.

The group $\Het^{1-b}({\overline B}, \overline\calg)$ has weights $>w+1-b$.
\end{lem}

\begin{proof} In the distinguished triangle (\ref{calgg}) introduced before Lemma \ref{lemkeylem} the perverse sheaf ${ j}_{!*}\calf$ is pure of weight $w$ by (\cite{bbd}, Corollaire 5.4.3) and  $\RR{ j}_{*}\calf$ is mixed of weights $\geq w$ by (\cite{bbd}, (5.1.14)). Moreover, (\cite{bbd}, 5.4.1 and 5.4.2) imply that $\calg$ is mixed of weights $>w$. From this the statement follows by applying again (\cite{bbd}, (5.1.14)), this time to the structural morphism $B\to\spec(k)$.
\end{proof}

We are now able to prove:

\begin{prop}\label{keyprop} When the base field $k$ is finite, the maps (\ref{pervpart1}) and (\ref{pervpart2}) defined in the previous section are isomorphisms.
\end{prop}

\begin{dem}  Applying Proposition \ref{keylem} with $\calf=\R^{2p-1}\pi_*\Q_\ell(p) [b]$  and $\calf=\R^{2q-1}\pi_*\Q_\ell(q) [b]$, respectively, we obtain that the maps in question are injective. We prove surjectivity  for (\ref{pervpart1}), the other case being similar. The map in question sits in the exact sequence
 $$
0\to
\Het^{1-b}(B,j_{!*}\R^{2p-1}\pi_*\Q_\ell(p) [b]) \to \Het^{1}(U,\R^{2p-1}\pi_*\Q_\ell(p)) \to H^{1-b}(B, \calg)
$$
with $\calg$ defined as in (\ref{defcalg}). We show that the last group here vanishes.
 Applying Lemma \ref{keylem2} with $\calf=\R^{2p-1}\pi_*\Q_\ell(p) [b]$ we obtain that $\Het^{1-b}(\overline B,\overline\calg)$ has nonzero weights (note that $w=b-1$ in this case); in particular, the invariants of Frobenius are trivial. On the other hand, we know from Lemma \ref{lemkeylem} (applied over $\kbar$) that $\Het^{-b}(\overline B,\overline\calg)=0$. Therefore the Hochschild--Serre sequence
 $$
 H^1(k, \Het^{-b}(\overline B,\overline\calg))\to H^{1-b}(B, \calg)\to H^0(k, \Het^{1-b}(\overline B,\overline\calg))
 $$
 shows that $H^{1-b}(B, \calg)=0$, as claimed.
\end{dem}

\begin{cor}\label{mainfinite}
Proposition \ref{mainprop} holds when $k$ is algebraic over a finite field.
\end{cor}

\begin{dem}
The case when $k$ is finite is immediate from the previous proposition. Otherwise, a class $\alpha^p_U\in\Het^{1}(U,\R^{2p-1}\pi_*\Q_\ell(p))$ as in Proposition \ref{mainprop} always comes from a class in some cohomology group $\Het^{1}(U_{0},\R^{2p-1}\pi_{0*}\Q_\ell(p))$ where $k_0\subset k$ is a large enough finite field over which $U$, $\pi$ and $\alpha^p_U$ are defined and $U$ (resp. $\pi$) arise by base change from $U_0$ (resp. $\pi_0$). Now apply the finite field case.
\end{dem}

\section{Proof of \ref{mainprop} via specialization}\label{secspec}

We now turn to arbitrary algebraically closed fields.
Suppose $k_0$ is an algebraically closed field, $K_0$ is the function field of a smooth $k_0$-variety $B_0$, and $X_0$ is a \mbox{smooth} projective $K_0$-variety that extends to a smooth projective morphism $\pi:\,{\mathcal X}_0\to U_0$ over an affine open subscheme $U_0\subset B_0$. As in $\S 6$ of \cite{bbd}, we may spread out the situation over the spectrum of a finite type $\Z$-subalgebra $A\subset k_0$ to an affine open immersion $j_A:\, U_A\to B_A$ and a smooth projective morphism $\pi_A:\,{\mathcal X}_A\to U_A$. Given a finite collection $\calf_1,\dots, \calf_r$ of constructible sheaves of $\Z/\ell\Z$-modules on $\X_A$ (with $\ell$ invertible on $A$) and another collection $\calg_1,\dots, \calg_s$ on $B_A$, we may find a dense open subscheme $S\subset\spec(A)$ such that the higher direct image sheaves ${R^q\pi_{S *}(\calf_i}|_{\X_{S}})$ and ${R^q j_{S *}(\calg_j}|_{\X_{S}})$ on $B_{S}$ are all constructible and commute with any base change $S'\to S$. This follows from the generic constructibility and base change theorem of Deligne (\cite{sga4.5}, Th. finitude, Theorem 1.9 or \cite{fu}, Theorem 9.3.1). Now as on (\cite{bbd}, p. 156) we find a strictly henselian discrete valuation ring $R$ with residue field the algebraic closure $\overline\F$ of a finite field and dominating the localization of $A$ at a closed point contained in $S$. Base changing from $A$ to $R$ we obtain morphisms $j_R:\, U_R\to B_R$ and $\pi_R:\,{\mathcal X}_R\to U_R$. Denoting by $\widetilde k$ an algebraic closure of the fraction field of $R$ we may base change the situation to obtain morphisms $\widetilde j:\, \widetilde U\to\widetilde B$ and $\widetilde\pi:\,\widetilde{\mathcal X}\to\widetilde  U$ of $\widetilde k$-schemes. Base changing to the closed point of $\spec(R)$ we obtain a situation with two morphisms $j:\, U\to B$ and $\pi:\,{\mathcal X}\to U$ over the algebraic closure of a finite field as in the previous section.

Now assume that $\X_0$, $U_0$ and $B_0$ are all equipped with stratifications $\T_{\X_0}$, $\T_{U_0}$, $\T_{B_0}$ with smooth geometrically connected strata and for each stratum $Z$ of any of these stratifications we are given a finite collection $\L_Z$ of locally constant irreducible sheaves of $\Z/\ell\Z$-modules satisfying conditions (a)--(c) in (\cite{bbd}, 2.2.10). We can then consider, as in (\cite{bbd}, 6.1.8), the full subcategory $D^b_{\T,\L}(\X_0,\Q_\ell)\subset D^b_c(X_0,\Q_\ell)$ of objects represented by a complex of $\Z_\ell$-sheaves whose reduction mod $\ell$ has cohomology sheaves which, when restricted to any stratum $Z$ of $\T_{\X_0}$, are locally constant and are iterated extensions of objects in $\L_{Z}$. We also consider the analogous subcategories $D^b_{\T,\L}(U_0,\Q_\ell)$ and $D^b_{\T,\L}(B_0,\Q_\ell)$ for $U_0$ and $B_0$, respectively, and assume that the stratifications are fine enough so that the functors $\R j_{0*}$ and $\R\pi_{0*}$ send the corresponding subcategories to each other. Now as in (\cite{bbd}, 6.1.8), after changing $A$ and $S$ if necessary we may spread out the additional data over $A$ such that the base change properties hold for cohomology sheaves of objects of the various categories $D^b_{\T,\L}$. Moreover, after base changing to the generic and special fibres of a well-chosen $R$ as above we obtain equivalences of triangulated categories
\begin{equation}\label{cateq}
D^b_{\T,\L}({\widetilde \X}, \Q_\ell)\leftrightarrow D^b_{\T,\L}(\X_R, \Q_\ell)\leftrightarrow D^b_{\T,\L}(\X, \Q_\ell)
\end{equation}
and similarly for $U_R$ and $B_R$ as in (\cite{bbd}, lemma 1.6.9). In addition, we may choose the stratifications so that these equivalences are compatible via the functors  $\R j_{R*}$ and $\R\pi_{R*}$ and their base changes, and are preserved by Grothendieck's six operations with respect to the maps $j_R$ and $\pi_R$. This latter fact is explained on (\cite{bbd}, p. 154).

The equivalences in (\ref{cateq}) are induced by natural pullback maps $u^*_\X:\, D^b_{\T,\L}(\X_R, \Q_\ell)\to D^b_{\T,\L}({\widetilde \X}, \Q_\ell)$ and $i^*_\X:\, D^b_{\T,\L}(\X_R, \Q_\ell)\rightarrow D^b_{\T,\L}(\X, \Q_\ell)$. Fixing a quasi-inverse ${(u_\X^*)}^{-1}$  for $u^*_\X$ and composing with $i^*_\X$ we obtain a specialization map
$$
sp_\X=i_\X^*{(u_\X^*)}^{-1}:\, D^b_{\T,\L}({\widetilde \X}, \Q_\ell)\to D^b_{\T,\L}(\X, \Q_\ell).
$$
Choosing quasi-inverses ${(u_U^*)}^{-1}$ and ${(u_B^*)}^{-1}$ compatibly we also get specialization maps $sp_U$ and $sp_B$. Up to modifying $A$ and $S$ one last time if necessary, we may assume that these specialization maps respect the perverse $t$-structures restricted to the subcategories above. This is again a consequence of the generic base change theorem, as explained on (\cite{bbd}, p. 154).

\begin{lem}
Assume $\widetilde\calf$ is a perverse sheaf in $D^b_{\T,\L}({\widetilde U}, \Q_\ell)$. There is a commutative diagram
$$
\begin{CD}
 j_{*!}sp_U\widetilde\calf @>>> \R j_*sp_U\widetilde\calf \\
 @VVV  @VVV\\
 sp_B{\widetilde j}_{*!}\widetilde\calf @>>> sp_B\R{\widetilde j}_*\widetilde\calf
\end{CD}
$$

\noindent of morphisms of perverse sheaves in $D^b_{\T,\L}({B}, \Q_\ell)$.
\end{lem}

\begin{dem}
Start with the commutative diagrams of morphisms of functors
\begin{equation}\label{2diag}
\begin{CD}
i_B^*\R j_{R!} @>>> \R j_!i^*_U \\
@VVV @VVV \\
i_B^*\R j_{R*} @>>> \R j_*i^*_U
\end{CD}
\qquad
\begin{CD}
u_B^*\R j_{R!} @>>> \R j_!u^*_U \\
@VVV @VVV \\
u_B^*\R j_{R*} @>>> \R j_*u^*_U.
\end{CD}
\end{equation}
Here the upper rows are base change maps for higher direct image functors with compact support, the lower rows are base change maps for usual higher direct images and the vertical maps are forget support maps. Commutativity follows from the construction of the compact support base change maps as in the proof of (\cite{fu}, 7.4.4 (i)). Applying the functor $sp_B=i_B^*{(u_B^*)}^{-1}$ to the second diagram on the left, the functor ${(u_U^*)}^{-1}$ to both diagrams on the right and splicing the resulting two diagrams together we obtain the commutative diagram
$$
\begin{CD}
\R j_!sp_U  @>>> \R j_*sp_U \\
@AAA @AAA \\
i_B^*\R j_{R!}{(u_U^*)}^{-1} @>>> i_B^*\R j_{R*}{(u_U^*)}^{-1}\\
@VVV @VVV \\
sp_B\R\widetilde j_!  @>>> sp_B\R\widetilde j_*.
\end{CD}
$$
Now the horizontal arrows of the first diagram of (\ref{2diag}) are isomorphisms because, as noted above, the equivalences induced by the functors $i_U^*$ and $i_B^*$ are compatible with the functors $\R j_*$ and $\R j_!$ (this, of course, ultimately boils down to Deligne's generic base change theorem). Hence we may invert the upper vertical arrows in the above diagram and obtain  a commutative diagram of functors:
\begin{equation}\label{specdiag}
\begin{CD}
\R j_!sp_U  @>>> \R j_*sp_U \\
@VVV @VVV \\
sp_B\R\widetilde j_!  @>>> sp_B\R\widetilde j_*
\end{CD}
\end{equation}
Now given a perverse sheaf $\widetilde\calf$ in $D^b_{\T,\L}({\widetilde U}, \Q_\ell)$, the specialization $sp_U\calf$ is again perverse as noted above, and perversity is also preserved by the functors $\R\widetilde j_*$ and $\R\widetilde j_!$ (\cite{kiwe}, Corollary III.6.2). Recalling that by definition $\widetilde j_{!*}\widetilde\calf\cong {\rm Im}({}^p{\mathcal H}^0(\R\widetilde j_!\widetilde\calf)\to {}^p{\mathcal H}^0(\R\widetilde j_*\widetilde\calf))$ and similarly for $j_{!*}sp_U\calf$, we see that the diagram induces a map
$$
j_{!*}sp_U\calf\to sp_B\widetilde j_{!*}\widetilde\calf
$$
using again that perversity is preserved under specialization. This is the left vertical map of the diagram in lemma. The right vertical map the same as in diagram (\ref{specdiag}). Finally, recalling that $\widetilde j^*\widetilde j_{!*}\calf=\calf$, we obtain the horizontal maps of the diagram of the lemma as an adjunction map and the specialization of an adjunction map, respectively, and commutativity results from the construction.
\end{dem}

Applying the functor $H^q(B,\_\_)$ to the above diagram we obtain:

\begin{cor}\label{corspec} In the situation of the lemma set $\calf:=sp_U\widetilde\calf$. There is a commutative diagram of cospecialization maps
$$
\begin{CD}
\Het^{q}(B,j_{!*}\calf) @>>> \Het^{q}(U,\calf)\\
@VVV @VVV \\
\Het^{q}(\widetilde B,{\widetilde j}_{!*}\widetilde\calf) @>>> \Het^{q}(\widetilde U,\widetilde\calf)
\end{CD}
$$
for all $q\geq 0$.
\end{cor}

Now let us return to our original algebraically closed base field $k_0$. The field $\widetilde k$ constructed above is algebraically closed of the same characteristic, so we may assume that at least one of the inclusions $k_0\subset \widetilde k$ or $\widetilde k\subset k_0$ holds. We stick to the first case; the second one is handled similarly.

\begin{lem}\label{bcdiag} Assume $k_0\subset \widetilde k$. For every perverse sheaf $\calf_0$ in $D^b(U_0, \Q_\ell)$ there is a commutative diagram
$$
\begin{CD}
\Het^{q}(B_0,j_{0!*}\calf_0) @>>> \Het^{q}(U_0,\calf_0)\\
@V{\cong}VV @VV{\cong}V \\
\Het^{q}(\widetilde B,{\widetilde j}_{!*}\widetilde\calf) @>>> \Het^{q}(\widetilde U,\widetilde\calf)
\end{CD}
$$
for all $q\geq 0$, where $\widetilde\calf$ is the pullback of $\calf_0$ to $\widetilde U$.
\end{lem}

\begin{dem}
Let $e:\,\spec(\widetilde k)\to \spec(k_0)$ be the natural morphism, so that $\widetilde\calf=e^*\calf_0$. The base change theorem for extensions of algebraically closed fields of characteristic prime to $\ell$ (see e.g. \cite{fu}, Corollary 7.7.3) gives us a commutative diagram
$$
\begin{CD}
\Het^{q}(B_0,j_{0!*}\calf_0) @>>> \Het^{q}(U_0,\calf_0)\\
@V{\cong}VV @VV{\cong}V \\
\Het^{q}(\widetilde B,e^*j_{0!*}\calf_0) @>>> \Het^{q}(\widetilde U,e^*\calf_0)
\end{CD}
$$
so it remains to identify the groups in the lower left corner. This follows by a similar, but simpler, argument as in the previous proof: in our construction the $\Z$-algebra $A$ and the open set $S$ were chosen so that the derived direct image functors (with or without compact support) commute with arbitrary base change and then the perverse $t$-structure is also preserved as on (\cite{bbd}, p. 158). The above argument {\em mutatis mutandis} then shows that the intermediate extension functor also commutes with the base change $e^*$.
\end{dem}

\begin{proof}[First proof of Proposition \ref{mainprop}] We return to the situation at the beginning of this section: $k_0$ is an algebraically closed field, $K_0$ is the function field of a smooth $k_0$-variety $B_0$, and $X_0$ is a \mbox{smooth} projective $K_0$-variety that extends to a smooth projective morphism $\pi:\,{\mathcal X}_0\to U_0$ over an affine open subscheme $U_0\subset B_0$. We moreover take a homologically trivial cycle class $\alpha_0^p\in {CH^p_{\rm hom}(X_0)}_\Q$, extend a representative cycle $z_0$ to a cycle $z_{U_0}$ over $U_0$ (changing $U_0$ if necessary) and take its cycle class ${\rm cl}(z_{U_0})\in\Het^{2p}(\X_0,\Q_\ell(p))$. Now we perform the spreading out procedure of the beginning of this section in such a way that we spread out $z_{U_0}$ over $A$ as well (modifying $A$ if necessary), base change it to a cycle $z_{U_R}$ on $\X_R$ and pull it back to a cycle $z_U$ on the special fibre $\X$. Note that under this procedure the cycle class ${\rm cl}(z_{U})\in\Het^{2p}(\X,\Q_\ell(p))$ cospecializes to ${\rm cl}(z_{U_0})$. This follows from the contravariant functoriality of the cycle class (usually proven via comparison with Chern classes as in \cite{milne}, Corollary V.10.7; note that Milne works over an algebraically closed field but the argument extends to a general base using the general theory of Chern classes as found in \cite{sga5}, Expos\'e VII).

Since $\alpha^p_0$ is homologically trivial, the class ${\rm cl}(z_{U_0})$ gives rise to a cohomology class $\alpha^p_{U_0}\in \Het^{1}(U_0,\R^{2p-1}\pi_{0*}\Q_\ell(p))$ via the truncation procedure recalled before Proposition \ref{mainprop}; similarly, ${\rm cl}(z_{U})$ gives rise to a cohomology class $\alpha^p_{U}\in \Het^{1}(U,\R^{2p-1}\pi_{*}\Q_\ell(p))$. Since truncations and pullbacks are compatible, the two classes correspond under the cospecialization map
$$
\Het^{1}(U,\R^{2p-1}\pi_*\Q_\ell(p))\to \Het^{1}(\widetilde U,\R^{2p-1}\widetilde\pi_*\Q_\ell(p))\cong \Het^{1}(U_0,\R^{2p-1}\pi_{0*}\Q_\ell(p))
$$
which exists by the generic base change theorem by our choice of the ring $R$. Now applying Corollary \ref{corspec} and Lemma \ref{bcdiag} with $\calf_0=\R^{2p-1}\pi_{0*}\Q_\ell(p)[b]$ and $q=1-b$ we obtain a commutative diagram
$$
\begin{CD}
\Het^{1-b}(B,j_{!*}\R^{2p-1}\pi_*{\Q}_\ell(p) [b]) @>>> \Het^{1}(U,\R^{2p-1}\pi_*{\Q}_\ell(p))\\
@VVV @VVV \\
\Het^{1-b}(B_0,j_{!*}\R^{2p-1}\pi_{0*}{\Q}_\ell(p) [b]) @>>> \Het^{1}( U_0,\R^{2p-1}\pi_{0*}{\Q}_\ell(p))
\end{CD}
$$
where the right vertical cospecialization map sends $\alpha^p_U$ to $\alpha^p_{U_0}$ as noted above. By Corollary \ref{mainfinite} the class $\alpha^p_U$ comes from a class in $\Het^{1-b}(B,j_{!*}\R^{2p-1}\pi_*\Q_\ell(p) [b])$. But then by commutativity of the diagram $\alpha^p_{U_0}$ comes from $\Het^{1-b}(B_0,j_{!*}\R^{2p-1}\pi_{0*}{\Q}_\ell(p) [b])$, as was to be shown.
\end{proof}

\section{Proof of \ref{mainprop} via decomposition}\label{secdec}

We begin with a general proposition that may be considered as an `absolute' version of Proposition \ref{mainprop}. Let again $k$ be a general algebraically closed field and $\X$ a smooth $k$-variety. Choose an open immersion $j_\X:\, \X\hookrightarrow\Y$ with dense image in a $k$-variety $\Y$. For what follows only the $k$-dimension of $\X$ will be relevant; let us denote it by $N$. Denote furthermore by $\ICY:=j_{\X!*}(\Q_\ell[N])$ the corresponding intersection complex. Recall that $\ICY=j_{\V!*}(\Q_\ell[N])$ for any other open immersion $j_\V:\, \V\hookrightarrow\Y$ with $\V$ nonempty and smooth.

\begin{prop}\label{ICclass}
The cycle class in $\Het^{2p}(\X,\Q_\ell(p))$ of a codimension $p$ cycle on $\X$ lies in the image of the restriction map $\Het^{2p-N}(\Y,\ICY(p))\to \Het^{2p}(\X,\Q_\ell(p))$.
\label{lcc}
\end{prop}

To prove the proposition, we invoke de Jong's theorem \cite{dejong} on alterations to find  a proper and generically \'etale morphism $\psi:\W\to \Y$  over $k$ such that $\W$ is regular.

\begin{lem}\label{lemic} The complex $\ICY$ identifies with a direct summand of
$\R\psi_*(\Q_{\ell}[N])$ in $D^b(\Y, \Q_\ell).$
\end{lem}

\begin{dem} By the decomposition theorem (\cite{bbd}, Theorem 6.2.5 together with note 77, p. 176) the complex $\R\psi_*(\Q_{\ell}[N])$ decomposes as a direct sum of shifts of  simple perverse sheaves. Fix a smooth open subscheme $\V\subset \Y$ such that the restriction $\psi_\V:\psi^{-1}(\V)\to \V$ of the alteration is finite \'etale, and denote by $j_\V:\, \V\to \Y$ the inclusion map. Applying the decomposition theorem to the perverse sheaf $\R\psi_{\V*}(\Q_{\ell}[N])$ on $\V$ and comparing the two decompositions we see that $j_{\V!*}\R\psi_{\V*}(\Q_{\ell}[N])$ is a direct summand of $\R\psi_{*}(\Q_{\ell}[N])$.

Now since $\psi_\V$ is finite \'etale, there is a map $\R\psi_{\V *}\psi^*_\V\Q_{\ell}\to\Q_{\ell}$ splitting the adjunction map $\Q_{\ell}\to \R\psi_{\V *}\psi_\V^*\Q_{\ell}$  up to multiplication by the degree of $\psi_\V$ according to the 'm\'ethode de la trace' (see e.g. Stacks Project, Tag 58.65). Observing that $\Q_{\ell}\cong  \psi_\V^*\Q_\ell$, we may thus identify $\Q_\ell$ with a direct summand of $\R\psi_{\V *}\Q_{\ell}$, whence $\ICY=j_{\V !*}(\Q_\ell[N])$ is a direct summand of $j_{\V!*}\R\psi_{\V*}(\Q_{\ell}[N])$. Putting this together with the previous paragraph we obtain the statement of the lemma.
\end{dem}

\begin{proof}[Proof of Proposition \ref{lcc}]
Let $c$ be a cycle class in $\Het^{2p}(\X,\Q_\ell(p))$, and let
$c_1$ be its image in $\Het^{2p}(\psi^{-1}(\X),\Q_\ell(p))$.
Let $c_2$ be an extension of $c_1$ to $\Het^{2p}(\W,\Q_\ell(p))$ (one obtains such an extension by taking the Zariski closure in $\W$ of a representative of $c_1$). Now let
$$
\lambda:\R\psi_*(\Q_{\ell}(p)[N])\to \ICY(p)
$$ be the projection map obtained from Lemma \ref{lemic} after twisting by $p$. Passing to cohomology over $\Y$ we obtain a map
$$
\rho:\Het^{2p-N}(\Y,\R\psi_*(\Q_{\ell}(p)[N]))\to \Het^{2p-N}(\Y,\ICY(p)(p))
$$
whose source may be identified with $\Het^{2p-N}(\W,\Q_{\ell}(p)[N])$. By restricting over $\X$ we obtain a commutative diagram
$$
\begin{CD}
\Het^{2p-N}(\W,\Q_{\ell}(p)[N]) @>{\rho}>> \Het^{2p-N}(\Y,\ICY(p)(p))\\
@V{j_\X^*}VV @VV{j_\X^*}V \\
\Het^{2p-N}(\psi^{-1}(\X),\Q_{\ell}(p)[N]) @>{\rho_\X}>> \Het^{2p-N}(\X,\Q_{\ell}(p)[N])
\end{CD}
$$
Now consider the pullback map
$$\psi_\X^*:\Het^{2p-N}(\X,\Q_{\ell}(p)[N])\to \Het^{2p-N}(\X,\R\psi_{\X*}\Q_{\ell}(p)[N])\cong  \Het^{2p-N}(\psi^{-1}(\X),\Q_{\ell}(p)[N])$$ where the first map is induced by the adjunction map $\Q_{\ell}(p)[N]\to \R\psi_{\X*}(\Q_{\ell}(p)[N])$ and $\psi_\X$ is the restriction of $\psi$ above $\X$. The composite map
$$
\rho_\X\circ\psi_\X^*:\, \Het^{2p-N}(\X,\Q_{\ell}(p)[N])\to \Het^{2p-N}(\X,\Q_{\ell}(p)[N])
$$
is induced by a map $\Q_{\ell}(p)[N]\to \Q_{\ell}(p)[N]$ in $D^b(\X, \Q_\ell)$ which is nonzero because so is its restriction to $\X\cap \V$ where $\V$ is as in the previous lemma (indeed, above $\X\cap \V$ it is multiplication by the degree of $\psi_\V$ by the 'm\'ethode de la trace' used above). This map comes, after shifting and twisting, from a map $\Q_\ell\to\Q_\ell$ of constant sheaves. Since $\X$ is integral, there is only one such map up to scaling. It follows that $c_1=\psi_\X^*(c)$ maps to a nonzero constant multiple of $c$ by $\rho_\X$. Since we also have $c_1=j_\X^*(c_2)$, by commutativity of the diagram we obtain that $c$ is in the image of the right vertical map
$j_\X^*:\Het^{2p-N}(\Y,\ICY(p)(p))\to \Het^{2p-N}(\X,\Q_{\ell,\X}(p)[N])$, which is what we wanted to prove.\end{proof}

\begin{proof}[Second proof of Proposition \ref{mainprop}] Suppose now $\Y=\X^B$ for a relative compactification $\pi^B:\,\X^B\to B$ of $\pi:\, \X\to U$, so that we have a commutative pullback diagram
$$
\begin{CD} \X @>{j_\X}>> \X^B \\
@V{\pi}VV @VV{\pi^B}V \\
U @>{j}>> B.
\end{CD}
$$
Consider the intersection complex $\ICXB=j_{\X !*}\Q_\ell[d+b]$ on $\X^B$. We have an adjunction map
\begin{equation}\label{globalcuki}
\ICXB(p)\to \R j_{\X*}j_\X^*\ICXB(p)\cong\R j_{\X*}\Q_\ell[d+b](p).
\end{equation}
Applying the functor $\R\pi^B_*$ and using the diagram above we obtain a map
\begin{equation}\label{cuki}
\R\pi^B_*\ICXB(p)\to \R\pi^B_*(\R j_{\X*}\Q_\ell[d+b](p))\cong\R j_*(\R\pi_{*}\Q_\ell[d+b](p)).
\end{equation}
By the decomposition theorem the complex $\R\pi^B_*\ICXB(p)$ in $D^b(B, \Q_\ell)$ is the direct sum of shifts of its perverse cohomology sheaves ${}^p\R^i\pi^B_*\ICXB(p)$. Restricting over $U$ we obtain the decomposition of the complex $\R\pi_*\Q_\ell[d+b](p)$ as the direct sum of shifts of its cohomology sheaves $\R^i\pi^B_*\ICXB(p)$. So the map above decomposes as a direct sum of the maps
$$
{}^p\R^i\pi^B_*\ICXB(p)[-i]\to \R j_*(\R^i\pi_{*}\Q_\ell[d+b](p))[-i].
$$
The map induced by the adjunction map (\ref{globalcuki}) on the $(2p-d-b)$-th cohomology of $\X^B$ is of the form
$$
H^{2p-d-b}(\X^B, \ICXB(p))\to H^{2p}(\X, \Q_\ell(p)).
$$
Using the map (\ref{cuki}) we may identify it with a map
$$
H^{2p-d-b}(B, \R\pi^B_*\ICXB(p))\to H^{2p}(U,\R\pi_{*}\Q_\ell(p))
$$
which thus decomposes as a direct sum of maps
\begin{equation}\label{cukidec}
H^{2p-d-b-i}(B, {}^p\R^i\pi^B_*\ICXB(p))\to H^{2p-i}(U,\R^i\pi_{*}\Q_\ell(p)).
\end{equation}
Now a cycle class in $H^{2p}(\X, \Q_\ell(p))$ comes from a class in $H^{2p-d-b}(\X^B, \ICXB(p))$ by Proposition \ref{ICclass}. If moreover it lies in the component $H^{2p-i}(U,\R^i\pi_{*}\Q_\ell(p))$, it comes from the component $H^{2p-d-b-i}(B, {}^p\R^i\pi^B_*\ICXB(p))$ of $H^{2p-d-b}(\X^B, \ICXB(p))$.

It remains to show that the map (\ref{cukidec}) factors through the map
$$
H^{2p-b-i}(B, j_{!*}\R^i\pi_*\Q_\ell(p)[b])\to H^{2p-i}(U,\R^i\pi_{*}\Q_\ell(p)),
$$
from which the proposition will follow by taking $i=2p-1$. To this end we decompose the perverse sheaf ${}^p\R^i\pi^B_*\ICXB(p)$ further into its simple components and compare it to the decomposition of the $(-d)$-th shift of its restriction $\R^i\pi_{*}\Q_\ell(p)[d+b]$ over $U$. It follows that ${}^p\R^i\pi^B_*\ICXB(p)$ decomposes as a direct sum of $j_{*!}(\R^i\pi_{*}\Q_\ell(p)[b])[d]$ and some simple components supported outside of $U$. These extra components vanish when restricted to $U$ and the claim is established.
\end{proof}

\begin{proof}[Sketch of proof of Theorem \ref{complextheo}]
Assume $k=\C$ and pick a class $\alpha^p\in CH^p_{\rm hom}(X)_\Q$. As before, we can extend it to a cycle class on some model $\X$ over $U\subset B$ suitably small, whence a topological cycle class $\alpha^p_t\in H^{2p}(\X(\C), \Q(p))$. Using the comparison theorem between complex and \'etale cohomology, we view the above group as a subgroup of $\Het^{2p}(\X, \Q_\ell(p))$. Similarly, we may view $H^{0}(U(\C), \R^{2p}\pi_*\Q(p))$ as a subgroup of $\Het^{0}(U, \R^{2p}\pi_*\Q_\ell(p))$, for instance by viewing $H^0$ in both the topological and the \'etale contexts as the subgroup of monodromy invariants. Thus homological triviality of $\alpha^p$ implies (using compatibility of Leray filtrations) that the class $\alpha^p_U\in\Het^{1}(U, \R^{2p-1}\pi_*\Q_\ell(p))$ considered in Proposition \ref{mainprop} comes from a topological class $\alpha^p_{U,t}\in H^{1}(U(\C), \R^{2p-1}\pi_*\Q(p))$, and similarly for $\alpha^q_U\in CH^q_{\rm hom}(X)_\Q$.

Next we consider the analogue of Proposition \ref{mainprop}: the class $\alpha^p_{U,t}$ lies in the image of the natural map $H^{1-b}(B(\C), j_{!*}\R^{2p-1}\pi_*\Q(p)[b])\to H^{1}(U(\C), \R^{2p-1}\pi_*\Q(p))$ and similarly for $\alpha^q_{U,t}$. This follows from the same arguments as in the proof above, using the topological version of the decomposition theorem as stated in \cite{decmi}. Note that the proof becomes simpler, since thanks to resolution of singularities one may take for $\X^B$ a regular compactification of $\X$ over $B$ (which may not, however, be smooth over $B$), and the introduction of $\W$ can be avoided. In particular, the analogue of Proposition \ref{ICclass} becomes obvious.

We construct the topological analogue of the pairing (\ref{cohpairing}) by exactly the same methods as in Section 2, using the intermediate extension functor and the formalism of Poincar\'e duality in the topological context (see e.g. \cite{iversen} for the latter). Finally, we need an analogue of the key injectivity statement of Proposition \ref{keylem} for topological sheaves of $\Q$-modules. It can be established in the same way as in the $\ell$-adic case. More quickly, it can be deduced from Proposition \ref{keylem} by base changing coefficients to $\Q_\ell$ (an injective operation, as above) and applying the comparison theorem.
\end{proof}

\noindent{\bf Challenge:} Show that for $B$ projective the pairing of Theorem \ref{complextheo} has values in the subgroup $H^{1,1}(B, \Q(1))$. The Lefschetz $(1,1)$ theorem would then imply a weaker form of Conjecture \ref{mainconj} over $k=\C$, namely that our cohomological pairing comes from a pairing $CH^p_{\rm hom}(X)_\Q\otimes CH^q_{\rm hom}(X)_\Q\to NS(B)_\Q$. We hope to return to this in a later version.

\section{The case of good reduction everywhere}\label{secgoodred}

In this section we present a simple geometric construction for the height pairing (and hence a solution of Conjecture \ref{mainconj}) in the case where the $K$-variety $X$ admits a smooth projective model $\X\to B$.

On the regular $k$-scheme $\X$ we have the  pairing
\begin{equation}\label{ip}
CH^p(\X)\times CH^q(\X)\to \pic(B),\,\, (z, z')\mapsto \langle z, z'\rangle
\end{equation}
obtained by composing the intersection pairing for $p+q=d+1$
$$
CH^p(\X)\times CH^q(\X)\to CH^{d+1}(\X), \,\, (z, z')\mapsto z\cdot z'
$$
with the pushforward map
$$
\pi_*:\, CH^{d+1}(\X)\to CH^1(B).
$$

The proof of the following proposition was inspired by the proof of a related statement in a preliminary version of \cite{kahn}.

\begin{prop}\label{homeq} Let $g:\, X\to\X$ be the natural map.
If $z\in CH^p(\X)$ satisfies $g^*z\in CH^p_{\rm hom}(X)$ and $z'\in CH^q(\X)$ satisfies $g^*z'=0$, then $\langle z, z'\rangle=0$.
\end{prop}

\begin{dem}
If $g^*z'=0$, there is some open subscheme $U\subset B$ such that $z'$ restricts to 0 in $CH^q(\X_U)$, where $\X_U:=\X\times_BU$. Denoting by $Z$ the complement of $U$ and setting $
\X_Z:=\X\times_BZ$, the localization sequence for Chow groups (\cite{fulton}, Proposition 1.8) implies that $z'$ comes from a class $z''$ in $CH^q(\X_Z)$. Let $Z_{\rm sing}\subset Z$ be the singular locus of $Z$. It is a closed subscheme of codimension $\geq 2$ in $B$, so $\pic(B)\stackrel\sim\to \pic(B\setminus Z_{\rm sing})$. Therefore we may replace $B$ by $B\setminus Z_{\rm sing}$ and $\X$ by $\X\times_B(B\setminus Z_{\rm sing})$ and assume $Z$, and hence $\X_Z$, are smooth. Let $\pi_Z:\X_Z\to Z$ be the morphism obtained from $\pi$ by base change, and let $\iota:\X_Z\to Z$ (resp. $\rho: Z\to B$) be the inclusion morphisms. We thus have a pullback diagram of $k$-varieties
$$
\begin{CD}
\X_Z @>{\iota}>> \X \\
@V{\pi_Z}VV @VV{\pi}V \\
Z @>{\rho}>> B
\end{CD}
$$
with $X_Z$ and $Z$ smooth but possibly disconnected.

We compute
$$
\langle z,z'\rangle=\pi_*(z\cdot z')=\pi_*(z\cdot\iota_*(z''))=
\pi_*(\iota_*(\iota^*(z)\cdot z''))=\rho_*(\pi_{Z*}(\iota^*(z\cdot z'')))
$$
where we used the projection formula (see \cite{fulton}, 8.1.1 (c)) in the third equality.

It suffices to show that the class $\pi_{Z*}(\iota^*(z)\cdot z'')\in CH^0(Z)$ vanishes. Since $CH^0(Z)$ is a finite direct sum of copies of $\Z$ indexed by the components of $Z$, it will be enough to show that
\mbox{$x^*(\pi_{Z*}(z_0'\cdot\iota^*(z)))$} vanishes for an arbitrary closed
point $x:{\rm Spec}\,k\to Z$. Denote by $\pi_x:\X_x\to \spec k$ the fibre of
$\pi$ over $x$ and let $\iota_x:X_x\to X_Z$ be the inclusion map, so that we have another pullback diagram
$$
\begin{CD}
\X_x @>{\iota_x}>> \X_Z \\
@V{\pi_x}VV @VV{\pi_Z}V \\
\spec k @>{x}>> Z
\end{CD}
$$
 Note that by (\cite{fulton}, Theorem 6.2 (a))  we have the base change compatibility $\pi_{x*}\circ\iota_x^*=x^*\circ\pi_{Z*}$. We can thus compute
$$
x^*(\pi_{Z*}(\iota^*(z)\cdot z''))=\pi_{x*}(\iota_x^*(\iota^*(z)\cdot z''))=\pi_{x*}(\iota_x^*(\iota^*(z))\cdot\iota_x^*(z'')).
$$
To show that the class on the right hand side is zero, it suffices to verify that its cycle class in $\Het^0(k(x),\Q_\ell(0))$ is 0. Now since $g^*z$ is homologically equivalent to zero, the restriction of $z$ to all fibres of $\pi$ is homologically equivalent to 0 by the smooth proper base change theorem. In particular, the cycle class of  $\iota_x^*(\iota^*(z))$ in $\Het^{2p}(\X_x,\Q_\ell(p))$ is trivial. The claim then follows from the compatibility of the cycle class map with push-forwards and products.
\end{dem}

\begin{cor}\label{goodredproduct}
The pairing (\ref{ip}) induces a pairing
$$
CH^p_{\rm hom}(X)\times CH^q_{\rm hom}(X)\to \pic(B)
$$
on generic fibres, still under the assumptions that $p+q=d+1$ and $X$ admits a smooth projective model $\X\to B$.
\end{cor}

\begin{dem}
Pick $\alpha^p\in CH^p_{\rm hom}(X)$ and $\alpha^q\in CH^q_{\rm hom}(X)$. Extend $\alpha^p$ to a cycle class on $\X$, for instance by taking the Zariski closure $z$ of a representative. Suppose that $z_1'$ and $z_2'$ are two cycles on $\X$ whose classes both restrict to $\alpha^q$ on $X$. Then $w:=z_1'-z_2'$ satisfies $g^*w=0$, and hence by the proposition we have $\langle z, w\rangle=0$. This shows that $\langle z, z_1'\rangle$ only depends on $\alpha^q$ and by symmetry the same is true of $\alpha^p$.
\end{dem}

\begin{rema}\rm
Inspection of the above arguments shows that the pairing of the corollary exists over an arbitrary perfect field $k$ (the only difference is that in the proof of Proposition \ref{homeq} the point $x$ should be a geometric point).
\end{rema}

It remains to check that the pairing in the above corollary is compatible with that of Theorem \ref{mainthm}. To do so, take representatives of $\alpha^p\in CH^p_{\rm hom}(X)$ and $\alpha^q\in CH^q_{\rm hom}(X)$ as above, extend them to cycles on the whole of $\X$ and consider the associated classes ${\alpha^p_B\in\Het^{1}(B,\R^{2p-1}\pi_*\Q_\ell(p))}$ and $\alpha^q_B\in\Het^{1}(B,\R^{2q-1}\pi_*\Q_\ell(q))$ as constructed before Proposition \ref{mainprop}. Since in this case $U=B$, all intermediate extensions are just identity maps and the pairing (\ref{cohpairing}) of Section 2 becomes the cup-product pairing
\begin{equation}\label{cohp2}
H^{1}(B, \R^{2p-1}\pi_*\Q_\ell(p))\times H^{1}(B, \R^{2q-1}\pi_*\Q_\ell(q))\to H^{2}(B, \Q_\ell(d+1))
\end{equation}
induced by fibrewise Poincar\'e duality. Thus the compatibility to be verified becomes:

\begin{prop}
The cycle class of the value of the pairing of Corollary \ref{goodredproduct} on the pair  $(\alpha^p,\alpha^q)$ equals the value of the pairing (\ref{cohp2}) on the pair $(\alpha^p_B,\alpha^q_B)$.
\end{prop}

\begin{dem} This is well known but we sketch an argument for the sake of completeness. To construct the pairing of Corollary \ref{goodredproduct} we extended $\alpha^p$ and $\alpha^q$ to $\X$ and took the intersection product (\ref{ip}) followed by the pushforward map to $\pic(B)$. The cohomological realization of this construction is the cup-product pairing
\begin{equation}\label{cupp}
H^{2p}(\X, \Q_\ell(p))\times H^{2q}(\X, \Q_\ell(q))\to H^{2d+2}(\X, \Q_\ell(d+1))
\end{equation}
followed by the pushforward map
\begin{equation}\label{push}
H^{2d+2}(\X, \Q_\ell(d+1))\to H^2(B, \Q_\ell(1))
\end{equation}
induced by $\pi$ that can be described in more detail as follows. We have an isomorphism
$$
H^{2d+2}(\X, \Q_\ell(d+1))\cong H^{2d+2}(B, \R\pi_*\Q_\ell(d+1))
$$
coming from the isomorphism of functors $\R\Gamma(\X,\_\_)=\R\Gamma(B,\_\_)\circ \R\pi_*$.

Since $\pi$ has relative dimension $d$, $\R\pi_*$ has trivial cohomology in degrees $>2d$, so that
$$
H^{2d+2}(B, \R\pi_*\Q_\ell(d+1))\cong H^{2d+2}(B, \tau_{\leq 2d}\R\pi_*\Q_\ell(d+1))
$$
where $\tau_\leq$ is the sophisticated truncation functor. The morphism
$$
\tau_{\leq 2d}\R\pi_*\Q_\ell(d+1)\to \R^{2d}
 \pi_*\Q_\ell(d+1)[-2d]
$$
gives a map
$$
H^{2d+2}(B,\tau_{\leq 2d}\R\pi_*\Q_\ell(d+1))\to H^2(B,\R^{2d}
 \pi_*\Q_\ell(d+1)).
$$
Finally, the trace map $\R^{2d}
 \pi_*\Q_\ell(d+1)\to \Q_\ell(1)$ of Poincar\'e duality induces a map
$$
H^2(B,\R^{2d}
 \pi_*\Q_\ell(d+1))\to H^2(B, \Q_\ell(1))
$$
and (\ref{push}) is the composition of these.

The tensor product pairing
$$
\Q_\ell(p)\otimes \Q_\ell(q)\to \Q_\ell(d+1)
$$
induces a derived pairing
\begin{equation}\label{dp}
\R\pi_*\Q_\ell(p)\otimes^{\bf L} \R\pi_*\Q_\ell(q)\to \R\pi_*\Q_\ell(d+1).
\end{equation}

Similarly, there is a derived pairing
\begin{equation}\label{dpgamma}
\R\Gamma(\X,\Q_\ell(p))\otimes^{\bf L} \R\Gamma(\X,\Q_\ell(p))\to \R\Gamma(\X,\Q_\ell(p+q))
\end{equation}
which induces (\ref{cupp}) by passing to cohomology groups. It is a formal exercise using Godement resolutions to verify that
the pairing (\ref{dpgamma}) coincides with the pairing
\begin{equation}\label{dpcomp}
\R\Gamma(B,\R\pi_*\Q_\ell(p))\otimes^{\bf L} \R\Gamma(B,\R\pi_*\Q_\ell(p))\to \R\Gamma(B,\R\pi_*\Q_\ell(p+q))
\end{equation}
induced by (\ref{dp}) via the composition $\R\Gamma(\X,\_\_)=\R\Gamma(B,\_\_)\circ \R\pi_*$.

To obtain the pairing (\ref{cohp2}) from (\ref{cupp}), we note first that (\ref{dp}) induces truncated pairings
$$\tau_{\leq i}\R\pi_*\Q_\ell(p)\otimes^{\bf L} \tau_{\leq j}\R\pi_*\Q_\ell(q)\to \tau_{\leq i+j}\R\pi_*\Q_\ell(d+1)$$
for all $i,j$. In particular, for $p+q=d+1$ we see that (\ref{dpcomp}) induces pairings
$$
H^{2p}(B,\tau_{\leq 2p-1}\R\pi_*\Q_\ell(p))\otimes^{\bf L} H^{2q}(B,\tau_{\leq 2q-1}\R\pi_*\Q_\ell(p))\to H^{2d+2}(B,\tau_{\leq 2d}\R\pi_*\Q_\ell(p+q))
$$
where we have seen that the last group maps to $H^2(B, \Q_\ell(1))$ via the trace map.

The commutative diagram
$$
\begin{CD}
\tau_{\leq 2p-2}\R\pi_*\Q_\ell(p)\otimes^{\bf L} \tau_{\leq 2q-1}\R\pi_*\Q_\ell(q) @>>> \tau_{\leq 2p-1}\R\pi_*\Q_\ell(p)\otimes^{\bf L} \tau_{\leq 2q-1}\R\pi_*\Q_\ell(q) \\
@VVV @VVV \\
\tau_{\leq 2d-1}\R\pi_*\Q_\ell(d+1) @>>> \tau_{\leq 2d}\R\pi_*\Q_\ell(d+1)
\end{CD}
$$
together with the distinguished triangles
$$
\tau_{\leq 2p-2}\R\pi_*\Q_\ell(p)\to \tau_{\leq 2p-1}\R\pi_*\Q_\ell(p)\to \R^{2p-1}\pi_*\Q_\ell(p)[-2p+1]\to \tau_{\leq 2p-2}\R\pi_*\Q_\ell(p)[1]
$$
and
$$
\tau_{\leq 2d-1}\R\pi_*\Q_\ell(d+1)\to \tau_{\leq 2d}\R\pi_*\Q_\ell(d+1)\to \R^{2d}\pi_*\Q_\ell(d+1)[-2d]\to \tau_{\leq 2d-1}\R\pi_*\Q_\ell(d+1)[1]
$$
shows that tensoring the first triangle by $\tau_{\leq 2q-1}\R\pi_*\Q_\ell(q)$ yields an induced pairing
$$
\R^{2p-1}\pi_*\Q_\ell(p)[-2p+1]\otimes^{\bf L} \tau_{\leq 2q-1}\R\pi_*\Q_\ell(q)\to \R^{2d}\pi_*\Q_\ell(d+1)[-2d].
$$
Since the target here is concentrated in degree $2d$, we see using the distinguished triangle
$$
\tau_{\leq 2q-2}\R\pi_*\Q_\ell(q)\to \tau_{\leq 2q-1}\R\pi_*\Q_\ell(q)\to \R^{2q-1}\pi_*\Q_\ell(q)[-2q+1]\to \tau_{\leq 2q-2}\R\pi_*\Q_\ell(q)[1]
$$
that the above pairing factors through a pairing
$$
\R^{2p-1}\pi_*\Q_\ell(p)[-2p+1]\otimes^{\bf L} \R^{2q-1}\pi_*\Q_\ell(q)[-2q+1]\to \R^{2d}\pi_*\Q_\ell(d+1)[-2d].
$$
Application of the functor $\R\Gamma(B,\_\_)$ finally induces
the pairing (\ref{cohp2}).

To sum up, the pairing (\ref{cohp2}) arises from the composition of (\ref{cupp}) and (\ref{push}) via a truncation procedure. It remains to note that, as explained before Proposition \ref{mainprop}, the cycle class of an extension of $\alpha^p$ (resp. $\alpha^q$) to $\X$ is sent exactly to $\alpha^p_B$ (resp. $\alpha^q_B$) under this procedure.
\end{dem}

\section{A conjecture about Arakelov Chow groups}\label{secarak}

In this section we formulate an Arakelovian analogue of Conjecture \ref{mainconj}.

Let $R$ be a regular arithmetic ring, i.e. a regular,
excellent, Noetherian domain, together
with a finite set $\mathcal S$ of injective ring homomorphisms $R\ra{\bf C}$
which is invariant under complex conjugation.
Consider a regular, integral scheme $B_0$ flat and of  finite type over  $R$.
We shall write
 $B_0({\bf C})$ for the set of complex points of the (non-connected) complex variety given by the disjoint union of the
$B_0\times_{R,\sigma}{\bf C}$ for all $\sigma\in\mathcal S$. It
naturally carries the structure of a complex manifold.

A hermitian line bundle on $B_0$ is by definition a line bundle $L$ together with a hermitian metric on the holomorphic line bundle associated with $L$ over $B_0({\bf C})$ which is invariant under complex conjugation.
We shall write $\widehat{\rm Pic}(B_0)$ for the group of isomorphism classes of hermitian line bundles on $B_0$, together with the group structure given by the tensor product (see \cite{gsar} for background and details). By construction, there is  a homomorphism
$\widehat{\rm Pic}(B_0)\to{\rm Pic}(B_0)$ which forgets the hermitian structure.

Fix an algebraic closure $k$ of the fraction field $k_0$ of $R$ and denote by $B$ the base change of $B_0$ to $k$.
We shall write $\phi:\widehat{\rm Pic}(B_0)\to
{\rm Pic}(B)$ for the composition of the above forgetful homomorphism with the natural pullback map
${\rm Pic}(B_0)\to {\rm Pic}(B)$. When $R$ is the ring of integers of a number field, $\mathcal S$ the set of all complex embeddings of $R$ and $B_0={\rm Spec}\,R$, there is a group homomorphism ${\rm deg}:\widehat{\rm Pic}({\rm Spec}(R)))\to{\bf R}$ called the arithmetic degree (see \cite{bogs}, 2.1.3). More generally, if $R$ and $\mathcal S$ are as before and $B_0$ is projective over $R$, then an ample hermitian line bundle $L$ on $B_0$ induces an arithmetic degree (or height) map $\widehat{\rm Pic}(B_0)\to{\bf R}$ by intersecting with the $(b-1)$-st power of the arithmetic first Chern class of $L$ and then applying the above degree map on $\widehat{\rm Pic}({\rm Spec}(R))$ (\cite{bogs}, 3.1.1).

Let $X_0$ be a smooth projective integral variety of dimension $d$ over the function field $K_0$ of $B_0$.

\begin{conj} Suppose that $p+q=d+1$.
\label{conjAr} There exists a pairing
$$
h:{\rm CH}^p_{\rm hom}({X_0})_\Q\otimes {\rm CH}^q_{\rm hom}({X_0})_\Q\to \widehat{\rm Pic}({B_0})_\Q
$$
with the following properties.

\begin{enumerate}
\item In the case when $R$ is the ring of integers of a number field, $\mathcal S$ the set of all complex embeddings of $R$ and $B_0={\rm Spec}\,R$, the pairing $\widehat{\rm deg}\circ h$ is the height pairing whose existence was conjectured by Beilinson in \cite{beilinsonhp}.

\item In the case when $B_{k_0}$ is geometrically integral denote by $X$ the base change of $X_0$ to $K:=kK_0$. Then the pairing $\phi\circ h$ is the composition of the pairing of Conjecture \ref{mainconj} with the base change map
$$
{\rm CH}^p_{\rm hom}({X_0})_\Q\otimes {\rm CH}^q_{\rm hom}({X_0})_\Q\to
{\rm CH}^p_{\rm hom}(X)_\Q\otimes {\rm CH}^q_{\rm hom}(X)_\Q.
$$
If moreover $B_0$ is projective, we expect $h$ to become a non-degenerate $\R$-valued pairing after base change to $\R$ and composition with the arithmetic degree map $\widehat{\rm Pic}({B_0})_\Q\to \R$ associated with an ample hermitian line bundle on $B_0$.
\end{enumerate}
\end{conj}

In the case of abelian varieties Moret-Bailly describes in \cite{moret}, III 4.4.1 a candidate for the above pairing in the case when $R$ is the ring of integers of a number field, $B={\rm Spec}\,R$ and $q=1$.

Finally, for $k=\C$ one may consider the space of complex $C^\infty$ differential $(1,1)$-forms $A^{1,1}(B)$ on $B$ and compose the pairing $h$ of Conjecture \ref{conjAr} with the map $\widehat{\rm Pic}({B_0})_\Q\to A^{1,1}(B)$  obtained by taking the curvature forms of hermitian line bundles. This would give rise to a pairing $${\rm CH}^p_{\rm hom}({X_0})_\Q\otimes {\rm CH}^q_{\rm hom}({X_0})_\Q\to A^{1,1}(B)$$ for which it would be interesting to have an analytic construction.

\end{document}